\documentclass[11pt]{article}
\usepackage{amsmath}
\usepackage{amsfonts}
\usepackage{amssymb}
\usepackage{mathrsfs}
\usepackage{amsthm}
\usepackage{palatino}
\usepackage[margin=2.5cm, vmargin={1.5cm}]{geometry}

\usepackage{dcolumn}

\usepackage{times}
\usepackage{graphicx}
\usepackage{xcolor}

\usepackage{pstricks}
\usepackage{pst-plot}

\renewcommand\footnotemark{}

\begin{document}

\title{Zeros of Jensen polynomials and  asymptotics for the  Riemann xi function}

\author{
  Cormac ~O'Sullivan\footnote{{\it Date:} Nov 29, 2020.
\newline \indent \ \ \
  {\it 2020 Mathematics Subject Classification:} 11M26, 11M06, 41A60.
  \newline \indent \ \ \
Support for this project was provided by a PSC-CUNY Award, jointly funded by The Professional Staff Congress and The City
\newline \indent \ \ \
University of New York.}
  }

\date{}

\maketitle

\def\s#1#2{\langle \,#1 , #2 \,\rangle}

\def\F{{\frak F}}
\def\C{{\mathbb C}}
\def\R{{\mathbb R}}
\def\Z{{\mathbb Z}}
\def\Q{{\mathbb Q}}
\def\N{{\mathbb N}}
\def\G{{\Gamma}}
\def\GH{{\G \backslash \H}}
\def\g{{\gamma}}
\def\L{{\Lambda}}
\def\ee{{\varepsilon}}
\def\K{{\mathcal K}}
\def\Re{\mathrm{Re}}
\def\Im{\mathrm{Im}}
\def\PSL{\mathrm{PSL}}
\def\SL{\mathrm{SL}}
\def\Vol{\operatorname{Vol}}
\def\lqs{\leqslant}
\def\gqs{\geqslant}
\def\sgn{\operatorname{sgn}}
\def\res{\operatornamewithlimits{Res}}
\def\li{\operatorname{Li_2}}
\def\lip{\operatorname{Li}'_2}
\def\pl{\operatorname{Li}}

\def\clp{\operatorname{Cl}'_2}
\def\clpp{\operatorname{Cl}''_2}
\def\farey{\mathscr F}

\newcommand{\stira}[2]{{\genfrac{[}{]}{0pt}{}{#1}{#2}}}
\newcommand{\stirb}[2]{{\genfrac{\{}{\}}{0pt}{}{#1}{#2}}}
\newcommand{\norm}[1]{\left\lVert #1 \right\rVert}

\newcommand{\e}{\eqref}


\newtheorem{theorem}{Theorem}[section]
\newtheorem{lemma}[theorem]{Lemma}
\newtheorem{prop}[theorem]{Proposition}
\newtheorem{conj}[theorem]{Conjecture}
\newtheorem{cor}[theorem]{Corollary}
\newtheorem{assume}[theorem]{Assumptions}
\newtheorem{adef}[theorem]{Definition}


\newcounter{counrem}
\newtheorem{remark}[counrem]{Remark}

\renewcommand{\labelenumi}{(\roman{enumi})}
\newcommand{\spr}[2]{\sideset{}{_{#2}^{-1}}{\textstyle \prod}({#1})}
\newcommand{\spn}[2]{\sideset{}{_{#2}}{\textstyle \prod}({#1})}

\numberwithin{equation}{section}

\let\originalleft\left
\let\originalright\right
\renewcommand{\left}{\mathopen{}\mathclose\bgroup\originalleft}
\renewcommand{\right}{\aftergroup\egroup\originalright}

\bibliographystyle{alpha}

\begin{abstract}
The classical criterion of Jensen for the Riemann hypothesis is that all of the associated Jensen polynomials have only real zeros. We find a new version of this criterion, using linear combinations of Hermite polynomials, and  show that this condition holds in many cases. Detailed asymptotic expansions  are given for the required Taylor coefficients of the xi function at $1/2$ as well as related quantities. These results build on those in the recent paper of Griffin, Ono, Rolen and Zagier.
\end{abstract}

\section{Introduction}
The {\em Riemann xi function},
$
  \xi(s) := {\textstyle \frac 12} s(s-1)\pi^{-s/2} \G(s/2)\zeta(s),
$
is entire of order $1$. It satisfies $\xi(1-s)=\xi(s)$ so that its development about the central point $1/2$ is
\begin{equation} \label{xip}
  \xi(s) = \sum_{n=0}^\infty \frac{\xi^{(2n)}(1/2)}{(2n)!} (s-1/2)^{2n}.
\end{equation}
The function $\Xi(z)$ is defined as $\xi(1/2+i z)$. In our context it is useful to define another function $$\Theta(z):=\xi(1/2+\sqrt{z})$$ which is  entire of order $1/2$. The Riemann hypothesis is equivalent to $\Xi(z)$  having only real zeros and to $\Theta(z)$ having only negative real zeros. The central point is now at $z=0$ and we write
\begin{equation} \label{xi1}
  \Theta(z) = \sum_{m=0}^\infty \frac{\g(m)}{m!} z^{m}, \qquad \qquad \g(m) :=  \frac{m!}{(2m)!} \xi^{(2m)}(1/2).
\end{equation}
Note that $\xi^{(2m)}(1/2)$ is positive for all $m$ since, as seen in \e{mom2}, it may be expressed as the integral of a positive function. Hence any real zeros of $\Theta(z)$ are necessarily negative. We also see that $\Theta(z)$ is {\em real} which just means it maps $\R$ into $\R$.
It is convenient to label a function as {\em hyperbolic} if all of its zeros are real.

Jensen describes  the results of his research into the zeros of functions in \cite{Je13}. For any real entire function $F(z)$ of genus at most $1$, he associated a family of polynomials, now called Jensen polynomials, and showed that they are all hyperbolic if and only if $F(z)$ is hyperbolic. He applied this idea to $\Xi(z)$ in \cite[p. 189]{Je13}, giving a criterion for the Riemann hypothesis, and developed further equivalent conditions for hyperbolicity in \cite{Je13} and unpublished work; see the discussion in \cite{Po27}.

Following \cite{DL,GORZ} we define the {\em Jensen polynomial of degree $d$ and shift $n$}  as
\begin{equation} \label{jensen}
  J^{d,n}(X):=\sum_{j=0}^d \binom{d}{j}\g(n+j) X^j.
\end{equation}
This is  associated to the $n$th derivative $\Theta^{(n)}(z)$, and
as we review in Corollary \ref{cdo},  $\Theta^{(n)}(z)$ is  hyperbolic  if and only if $J^{d,n}(X)$  is hyperbolic for all $d \gqs 1$.  Also, if $\Theta(z)$ is hyperbolic then all of its derivatives must be hyperbolic as well; see Corollary \ref{cdo2}. Hence we obtain the following extended criterion which has presumably been known since the time of  Jensen.

\begin{theorem} \label{rh}
The Riemann hypothesis is true if and only if  $J^{d,n}(X)$ is   hyperbolic for all $d \gqs 1$, $n \gqs 0$.
\end{theorem}

Griffin, Ono, Rolen and Zagier  revived interest in Theorem \ref{rh} when they showed in \cite{GORZ} that a great many of the Jensen polynomials $J^{d,n}(X)$ are hyperbolic. For every fixed $d$, \cite[Thm. 1]{GORZ} states that $J^{d,n}(X)$ is  hyperbolic for all sufficiently large $n$. This is made more explicit in \cite{Ono} where their Theorem 1.1 shows that $J^{d,n}(X)$ is  hyperbolic whenever $n \gg e^{8d/9}$.
In these papers, the results are demonstrated by using precise asymptotics for $\g(m)$ to show that  renormalized versions of $J^{d,n}(X)$ may be approximated by Hermite polynomials $H_d(X)$ as $n\to \infty$.

We  give a variant of Theorem \ref{rh} next by bringing in Hermite polynomials from the beginning. Set
\begin{equation} \label{jen-herm}
  P^{d,n}(X):=\sum_{j=0}^d \binom{d}{j}\g(n+j) H_{d-j}(X).
\end{equation}

\begin{theorem} \label{hyp}
The Riemann hypothesis is true if and only if  $P^{d,n}(X)$ is   hyperbolic for all $d \gqs 1$, $n \gqs 0$.
\end{theorem}

In fact, we see in Sect. \ref{jkl} that Theorem \ref{hyp} is a special case of a more general result where the Hermite polynomials in  \e{jen-herm} may be replaced by any Jensen polynomial associated to an element of the Laguerre-P\'olya class.

Checking that the zeros of $P^{d,n}(X)$ are real seems easier than for $J^{d,n}(X)$. Combining the asymptotics of $\g(m)$ with a tailor-made theorem of Tur\'an (that was also employed in \cite{Ono})  gives the next result directly, showing that $P^{d,n}(X)$ has only real zeros for all but a relatively small number of shifts $n$.

\begin{theorem} \label{chem}
For all $d$ sufficiently large,  $P^{d,n}(X)$  is hyperbolic whenever $n/\log^2 n \gqs d^{3/4}/2$.
\end{theorem}

Comparing $J^{d,n}(X)$ and $P^{d,n}(X)$ for specific $d$, $n$, we will see in Corollary \ref{poly-herm2} that $J^{d,n}(X)$ being hyperbolic implies that  $P^{d,n}(X)$ is  hyperbolic. Hence the results of \cite{GORZ,Ono} also apply to $P^{d,n}(X)$. Chasse in \cite[Sect. 3]{Cha13}, proved that $J^{d,n}(X)$  is  hyperbolic\footnote{This range was extended in \cite[Cor. 1.3]{Ono} and may be further extended to $d \lqs 9\times 10^{24}$ using \cite{PT}.} for all $n\gqs 0$ and $d \lqs 2\times 10^{17}$ and so the same is true for  $P^{d,n}(X)$. As an example, it is easy to see that $J^{2,n}(X)$ is hyperbolic if and only if
\begin{equation} \label{tu}
  \g(n+1)^2 \gqs \g(n)\g(n+2).
\end{equation}
This is the  Tur\'an inequality, necessary for the Riemann hypothesis, and first proved for all $n \gqs 0$ in \cite{Cso86}. However, $P^{2,n}(X)$ is hyperbolic if and only if
\begin{equation} \label{tu2}
  \g(n+1)^2 + 2\g(n)^2 \gqs \g(n)\g(n+2),
\end{equation}
and clearly \e{tu} implies \e{tu2} but not the other way around. 

The key ingredient in the proof of Theorem \ref{chem} is the asymptotic expansion of the coefficients $\g(n)$ in \e{xi1} for large $n$. This is shown in \cite[Thm. 9, Eq. (14)]{GORZ} with an application of Laplace's method. We give a  more precise version of this result by including the usual error estimates and giving formulas for all the coefficients. We also confirm a suggestion of Romik in \cite[Sect. 6.1]{Rom} that the answer can be conveniently expressed in terms of $W(2n/\pi)$, with  $W$ the Lambert function.  Recall that this function is the inverse to $x \mapsto xe^x$ and so satisfies
\begin{equation} \label{lamb}
  W(xe^x)=x, \qquad W(x) e^{W(x)}=x
\end{equation}
for at least $x\gqs 0$. It is non-negative and increasing
for $x\gqs 0$ and  $W(x)\lqs \log x$ holds for $x\gqs e$.

\begin{theorem} \label{mainthm3}
 Set $w :=W(2n/\pi)$. Then as $n \to \infty$ we have
\begin{equation*}
 \g(n) = 4 \pi^{2} e^{7w/4}  \sqrt{\frac {w}{w+1}} \left(\frac{e w^2}{16 n e^{2/w}} \right)^n \left( 1+  \sum_{k=1}^{K-1}\frac{c_k(w)}{n^k}+  O\left( \frac{\log^K(n)}{n^K}\right) \right)
\end{equation*}
for an implied constant depending only on $K$. Each $c_k(w)$ is a rational function of $w$ with size $O(\log^k(n))$ and given explicitly in \e{stx2}.
\end{theorem}

The first coefficient is
\begin{equation} \label{c1w}
  c_1(w) = -\frac{w^4+58 w^3+29 w^2-24 w-16}{192 (w+1)^3}.
\end{equation}
Table \ref{xeb} compares the approximations of Theorem \ref{mainthm3} with the actual value of $\g(n)$ for $n=1000$. All decimals are correct to the accuracy shown. See also  Table 2 in \cite{GORZ}, (their $\g(n)$ is $8$ times larger).

\begin{table}[ht]
\centering
\begin{tabular}{ccc}
\hline
 $K$   & Theorem \ref{mainthm3} & \\
\hline
 $1$    & $4.84\textcolor{gray}{60204243211378239} \times 10^{-2568}$ & \\
 $3$    & $4.845042611\textcolor{gray}{1532216799} \times 10^{-2568}$ & \\
 $5$    & $4.84504261127258\textcolor{gray}{84216} \times 10^{-2568}$ & \\
 $7$    & $4.8450426112725879772\textcolor{gray}{} \times 10^{-2568}$  &  \\
\hline
\phantom{$\g(1000)$} & $4.8450426112725879772 \times 10^{-2568}$ & $\g(1000)$\\
\hline
\end{tabular}
\caption{The approximations of Theorem \ref{mainthm3} to $\g(1000)$.} \label{xeb}
\end{table}

Theorem \ref{mainthm3} follows from the next theorem, giving the asymptotics of $\xi^{(2n)}(1/2)$ from \e{xip}.

\begin{theorem} \label{mainthm2}
 Set $w :=W(2n/\pi)$. Then as $n\to \infty$ we have
\begin{equation} \label{ytr}
  \xi^{(2n)}(1/2)= 4\pi^{2}e^{7w/4}\sqrt{\frac{2 w}{w+1}} \left( \frac{w}{2e^{1/w}}\right)^{2n}
 \left( 1+  \sum_{k=1}^{K-1}\frac{\mu_k(w)}{n^k}+  O\left( \frac{\log^{K}(n)}{n^K}\right) \right)
\end{equation}
for an implied constant depending only on $K$. Each $\mu_k(w)$ is a rational function of $w$ with size $O(\log^k(n))$ and given explicitly by \e{same2} and \e{pq3x}.
\end{theorem}

Theorem \ref{mainthm2} is essentially a reformulated  version of \cite[Thm. 9]{GORZ} where they used the solution $L$ to the equation $n=L(\pi e^L+3/4)$ instead of $W(2n/\pi)$. The main term of the expansion \e{ytr} appears in Thm. 6.1 of \cite{Rom} and may also be compared with the weaker results of \cite{Pu} and \cite{Co}.

The asymptotics of many related integrals are covered by our techniques in Sect. \ref{lap} and a general result is formulated  in Theorem \ref{ian2}. This  gives the complete asymptotic expansion as $n \to \infty$ of
\begin{equation} \label{iafn}
  I_\alpha(f;n):=\int_1^\infty  (\log t)^n e^{-\alpha t}  f(t) \, dt \qquad \qquad (\alpha >0)
\end{equation}
for suitable functions $f$. For an application of this, also involving  Hermite polynomials, recall Tur\'an's expansion
\begin{equation} \label{turanbn}
  \Xi(z)=\sum_{n=0}^\infty (-1)^n b_{2n} H_{2n}(z)
\end{equation}
for $\Xi(z)= \xi(1/2+i z)$. This series converges locally uniformly  in $\C$ according to \cite[Thm. 2.1]{Rom}. We may extend the asymptotics for $b_{2n}$ in \cite[Thm. 2.7]{Rom} with the next result.

\begin{theorem} \label{turanb}
 Set $w :=W(2n/\pi)$. Then as $n \to \infty$ we have
\begin{equation*}
  b_{2n}=4\pi^{2} \frac{ e^{7w/4-w^2/16}}{(2n)!}\sqrt{\frac{ 2w}{w+1}} \left( \frac{w}{4e^{1/w}}\right)^{2n}
 \left( 1+  \sum_{k=1}^{K-1}\frac{\tau_k(w)}{n^k}+  O\left( \frac{\log^{3K}(n)}{n^K}\right) \right)
\end{equation*}
for an implied constant depending only on $K$. Each $\tau_k(w)$ is a rational function of $w$ with size $O(\log^{3k}(n))$ and given by the formulas \e{same}, \e{pq2} and \e{pq3}.
\end{theorem}

We remark that Hermite polynomials have also appeared recently in connection with the asymptotics of $\zeta(s)$ in \cite{OSrs}, which gives a generalization of the Riemann-Siegel formula.
Romik explores    further orthogonal polynomial expansions of $\Xi[z]$ in  \cite{Rom}.
Wagner, in \cite{Wa}, extends the work in  \cite{GORZ} to large classes of $L$-functions.
 The techniques in this paper should also be useful for these generalizations.

\vskip 4mm
{\bf Acknowledgements.} Thanks to Jacques G\'elinas, Dan Romik, Tim Trudgian and both referees for their helpful comments.

\section{Preliminaries}
The {\em Hermite polynomials} have generating function $\exp(2Xt-t^2)$ and the explicit expression
\begin{equation*}
  H_d(X)= d! \sum_{r=0}^{\lfloor d/2\rfloor} \frac{ (-1)^r }{r! (d-2r)!}  (2X)^{d-2r}.
\end{equation*}
These polynomials are a special case of the Laguerre polynomials appearing in \e{lague}.

The {\em partial ordinary Bell polynomials} $\hat{\mathcal B}_{i,j}$ are  useful devices to keep track of power series  coefficients. With $j \in \Z_{\gqs 0}$, we have the generating function definition
\begin{equation} \label{pobell2}
    \left( p_1 x +p_2 x^2+ p_3 x^3+ \cdots \right)^j = \sum_{i=j}^\infty \hat{\mathcal B}_{i,j}(p_1, p_2, p_3, \dots) x^i.
\end{equation}
Clearly $\hat{\mathcal B}_{i,0}(p_1, p_2, p_3,  \dots) = \delta_{i,0}$.
The formulas
\begin{align} \label{pobell}
    \hat{\mathcal B}_{i,j}(p_1, p_2, p_3, \dots) & = \sum_{\substack{1\ell_1+2 \ell_2+ 3\ell_3+\dots = i \\ \ell_1+ \ell_2+ \ell_3+\dots = j}}
    \frac{j!}{\ell_1! \ell_2! \ell_3! \cdots } p_1^{\ell_1} p_2^{\ell_2} p_3^{\ell_3} \cdots, \\
& = \sum_{n_1+n_2+\dots + n_j = i}
    p_{n_1}p_{n_2} \cdots p_{n_j} \qquad \qquad (j \gqs 1) \label{pobell3}
\end{align}
hold, where the sum  in \e{pobell} is over all possible $\ell_1$, $\ell_2$, $\ell_3, \dots \in \Z_{\gqs 0}$ and the sum in \e{pobell3} is over all possible $n_1$, $n_2,  \dots \in \Z_{\gqs 1}$.
 For instance,
\begin{equation*}
  \hat{\mathcal B}_{9,6}(p_1, p_2, p_3, \dots) = 20 p_1^3 p_2^3+30 p_1^4 p_2 p_3 + 6 p_1^5 p_4.
\end{equation*}
For $j \gqs 1$ we see from \e{pobell3} that $\hat{\mathcal B}_{i,j}(p_1, p_2, p_3, \dots)$ is a polynomial in $p_1, p_2, \dots, p_{i-j+1}$ of homogeneous degree $j$ with positive integer coefficients.
See the discussion and references in \cite[Sect. 7]{OSper}, for example, for more information.

As an application we will need later, consider
\begin{equation} \label{hot}
  h_u(x) := \log\left( 1+\frac{\log(x+1)}{u}\right) = \sum_{i=1}^\infty \ell_i(u) x^i
\end{equation}
which is a holomorphic function of $x\in \C$ for  $|x|\lqs 1/2$ and $u\gqs 2$, say.
To find the coefficients $\ell_i(u)$ write
\begin{align}
   \log\left( 1+\frac{\log(x+1)}{u}\right) & = \sum_{j=1}^\infty \frac{(-1)^{j+1}}{j \cdot u^j} \log^j(x+1) \notag\\
  & = \sum_{j=1}^\infty \frac{(-1)^{j+1}}{j \cdot u^j} \sum_{i=j}^\infty \hat{B}_{i,j}( 1,-{\textstyle \frac 12},{\textstyle \frac 13},\dots) x^i\notag\\
& = \sum_{i=1}^\infty x^i  \sum_{j=1}^i  \hat{B}_{i,j}( 1,-{\textstyle \frac 12},{\textstyle \frac 13},\dots)\frac{(-1)^{j+1}}{j \cdot u^j}. \label{wash}
\end{align}
Then
\begin{equation}
  \ell_i(u)  = \sum_{j=1}^i \hat{B}_{i,j}( 1,-{\textstyle \frac 12},{\textstyle \frac 13},\dots) \frac{(-1)^{j+1}}{j \cdot u^j} \label{covi}
\end{equation}
with
\begin{equation}
 \ell_1(u)=\frac 1u, \qquad \ell_2(u)=-\frac 1{2u}-\frac 1{2u^2}, \qquad \ell_3(u)=\frac 1{3u}+\frac 1{2u^2}+\frac 1{3u^3},  \qquad \text{etc}.\label{luex}
\end{equation}

We also record here a basic bound for the incomplete gamma function $\G(s,a):=\int_a^\infty e^{-x} x^{s-1}\, dx$.

\begin{lemma} \label{inc-gam}
For $a,r \gqs 0$ we have
\begin{equation*}
  \G(r+1,a) \lqs 2^r(a^r + \G(r+1))e^{-a}.
\end{equation*}
\end{lemma}
\begin{proof}
The result follows from
\begin{align*}
  \int_a^\infty e^{-x} x^{r}\, dx & = e^{-a}\left(\int_0^a e^{-x}(x+a)^r \, dx +  \int_a^\infty e^{-x}(x+a)^r \, dx\right) \\
   & \lqs e^{-a}\left(\int_0^\infty e^{-x}(2a)^r \, dx +  \int_0^\infty e^{-x}(2x)^r \, dx\right). \qedhere
\end{align*}
\end{proof}
Hence, for all $a,r \gqs 0$ and $c>0$
\begin{equation}\label{gmmb}
  \int_a^\infty e^{-c x} x^{r}\, dx \ll c^{-r-1} \left((ac)^r +1\right) e^{-ac}
\end{equation}
for an implied constant depending only on $r$.

\section{Jensen polynomials and the Laguerre-P\'olya class} \label{jkl}
\subsection{Background}
In \cite{Po27} P\'olya defines a function  to be of {\em genus $1^*$} (``erh\"{o}htem
Genus $1$'') if it is of the form $e^{-\alpha z^2}f(z)$ for $\alpha \gqs 0$ and $f(z)$  an entire function of genus at most $1$. 
So this class includes some entire functions of order $2$ and, by Hadamard's theorem, all entire functions of order $<2$.
It is easy to see that the product of two functions of genus $1^*$ is also of genus $1^*$. See for example \cite[Sect. 1]{kk} for a discussion of P\'olya's long study of genus $1^*$ functions.

The {\em Laguerre-P\'olya class} consists of functions that are real, entire of  genus $1^*$ and that satisfy the added condition of having only real zeros. In particular, the functions $\Xi(z)$ and $\Theta(z)$ are both real entire functions of  genus $1^*$ and they are in the Laguerre-P\'olya class if and only if the Riemann hypothesis is true.
P\'olya and Schur in \cite{PS14} characterize the  Laguerre-P\'olya class in various ways and we give their description in terms of Jensen polynomials next.
For a formal power series $\Phi(z)=\sum_{j=0}^\infty c_j z^j/j!$  define
\begin{equation*}
  g_d(\Phi;x):= \sum_{j=0}^d \binom{d}{j} c_j x^j.
\end{equation*}
to be the {\em Jensen polynomial of degree $d$ associated to $\Phi$}. Let
\begin{equation*}
  g^*_d(\Phi;x):= x^d g_d(\Phi;1/x) = \sum_{j=0}^d \binom{d}{j} c_j x^{d-j}
\end{equation*}
be the reciprocal polynomial.
We have easily
\begin{equation} \label{appe}
  \frac d{dx} g_d(\Phi;x) = d \cdot g_{d-1}(\Phi';x), \qquad \frac d{dx} g^*_d(\Phi;x) = d \cdot g^*_{d-1}(\Phi;x)
\end{equation}
and the right identity in \e{appe} indicates that $g^*_d(\Phi;x)$ is an {\em Appell polynomial}. More properties of Jensen polynomials are given in \cite[Prop. 2.1]{Cso90}.

\begin{theorem}[P\'olya-Schur  \cite{PS14}] \label{psthm}
Let $\Phi(z)=\sum_{j=0}^\infty c_j z^j/j!$ be a formal power series with real coefficients. Then $\Phi(z)$ converges uniformly on compact sets in $\C$ to an entire function in the Laguerre-P\'olya class if and only if $g_d(\Phi;z)$ is hyperbolic for all $d\gqs 1$.
\end{theorem}

Laguerre had seen cases of this theorem and Jensen proved it in \cite[pp. 183--187]{Je13} with the assumption that $\Phi(z)$ is a real entire function of genus at most $1$.  P\'olya showed in \cite[Sect. 8]{Po27} that Jensen's proof  extends easily to $\Phi(z)$ being of genus $1^*$ and we describe this briefly next.

\begin{proof}[Sketch of proof when $\Phi(z)$ is of genus $1^*$]
In the easier direction, assume all $g_d(\Phi;z)$ are hyperbolic. Then since $g_d(\Phi;z/d) \to \Phi(z)$ locally uniformly  in $\C$ as $d\to \infty$, it follows from Hurwitz's theorem for example that $\Phi(z)$  must be hyperbolic.

Let $D=\frac{d}{dz}$. Then it can be shown to follows from Rolle's theorem that if $p(z)$ is a real hyperbolic polynomial then $(D-a)p(z)=p'(z)-a p(z)$ is also hyperbolic for all $a\in \R$. If $f(z)$ and $p(z)$ are  real hyperbolic polynomials then applying this argument $\deg(f)$ times shows $f(D)p(z)$ is hyperbolic. Assume $\Phi(z)$ is hyperbolic. From its Weierstrass factorization it is possible to construct real hyperbolic polynomials $\Phi_n(z)$ that converge uniformly to $\Phi(z)$ as $n\to \infty$. Then each $\Phi_n(D) p(z)$ is hyperbolic and as $n\to \infty$ they converge to $\Phi(D)p(z)$. It follows that $\Phi(D)p(z)$ is hyperbolic whenever $p(z)$ is. The final step is to note that $\Phi(D)z^d = g^*_d(\Phi;z)$, making $g^*_d(\Phi;z)$ and hence $g_d(\Phi;z)$ hyperbolic.
\end{proof}

 The simple idea behind our Theorem \ref{hyp} is to see what happens when $p(z)=z^d$, in the last sentence of the above proof, is replaced by $H_d(z)$. Also note that a similar limiting procedure allows the polynomial $p(z)$ to be replaced by more general functions; see Theorem \ref{prodj}. The full Theorem \ref{psthm}, requiring no conditions on $\Phi(z)$, is stated in \cite[p. 111]{PS14} and proved there in Sections 2--4. See for example \cite[Thm. 2.7]{Cso89} for further characterizations of the Laguerre-P\'olya class.

In our previous notation, with \e{xi1} and \e{jensen},
\begin{equation} \label{jg}
  J^{d,n}(x)=g_d(\Theta^{(n)};x)
\end{equation}
since the Taylor coefficients of the $n$th derivative $\Theta^{(n)}(z)$ are shifted by $n$.

\begin{cor} \label{cdo}
The function $\Theta^{(n)}(z)$ is hyperbolic if and only if $J^{d,n}(z)$  is hyperbolic for all $d \gqs 1$.
\end{cor}
\begin{proof}
We know that $\Theta(z)$ is real and entire of order $1/2$. All of its derivatives must have the same properties and this implies that   $\Theta^{(n)}(z)$ is  real and entire of genus $1^*$ for all $n\gqs 0$. Therefore $\Theta^{(n)}(z)$ is hyperbolic if and only if it is  the Laguerre-P\'olya class.
The corollary now follows  from \e{jg} and Theorem \ref{psthm}.
\end{proof}

\begin{cor} \label{cdo2}
If $\Theta^{(n)}(z)$ is hyperbolic then $\Theta^{(m)}(z)$ is also hyperbolic for all $m>n$.
\end{cor}
\begin{proof}
By Corollary \ref{cdo}, if $\Theta^{(n)}(z)$ is hyperbolic  then $J^{d,n}(z)$ is hyperbolic for all $d \gqs 1$. The  identity on the left of \e{appe} implies that
\begin{equation*}
  \frac{d}{dz} J^{d,n}(z) = d \cdot J^{d-1,n+1}(z).
\end{equation*}
As we saw in the proof of Theorem \ref{psthm}, the derivative of a real hyperbolic polynomial must be hyperbolic. Therefore $J^{d,n+1}(z)$ is hyperbolic for all $d \gqs 1$.
 Hence, by Corollary \ref{cdo} again, $\Theta^{(n+1)}(z)$ is  hyperbolic. Induction completes the proof.
\end{proof}

\subsection{Jensen polynomials of products}

\begin{lemma} \label{wim}
Let $\Phi$ and $\Omega$ be two formal power series with  $\Phi(z)=\sum_{j=0}^\infty c_j z^j/j!$. Then for all $d\gqs 0$
\begin{equation} \label{iran}
  g_d^*(\Phi \cdot \Omega;x) = \sum_{j=0}^d \binom{d}{j} c_j \cdot g^*_{d-j}(\Omega;x).
\end{equation}
\end{lemma}
\begin{proof}
This is an easy exercise.
\end{proof}
Replacing $x$ by $1/x$ in \e{iran} also yields
\begin{equation} \label{ave}
  g_d(\Phi \cdot \Omega;x) = \sum_{j=0}^d \binom{d}{j} c_j x^j \cdot g_{d-j}(\Omega;x).
\end{equation}

\begin{theorem} \label{prod}
Let $\Phi(z)=\sum_{j=0}^\infty c_j z^j/j!$ be a real entire function of genus $1^*$ and let $\Omega(z)$ be in the Laguerre-P\'olya class. Then $\Phi(z)$ is hyperbolic if and only if the polynomials
\begin{equation} \label{ran}
   \sum_{j=0}^d \binom{d}{j} c_j \cdot g^*_{d-j}(\Omega;x)
\end{equation}
 are hyperbolic for all $d\gqs 1$.
\end{theorem}
\begin{proof}
First note that $\Phi \cdot \Omega$ is of genus $1^*$ because each of $\Phi$ and  $ \Omega$ are. The zeros of $ \Omega$ are real since it is in the Laguerre-P\'olya class. Therefore $\Phi \cdot \Omega$ is in the Laguerre-P\'olya class if and only if $\Phi(z)$ is hyperbolic. By Theorem \ref{psthm}, $\Phi \cdot \Omega$ is in the Laguerre-P\'olya class if and only if $g_d^*(\Phi \cdot \Omega;x)$ is hyperbolic for all $d\gqs 1$. Noting that $g_d^*(\Phi \cdot \Omega;x)$ equals \e{ran} by Lemma \ref{wim} finishes the proof.
\end{proof}

Clearly we could replace \e{ran} by \e{ave} in the statement of Theorem \ref{prod} and get the same result.
The function
$
  e^{-z^2} = 1-\frac{2!}{1} \cdot \frac{z^2}{2!} + \frac{4!}{2!} \cdot \frac{z^4}{4!}- \dots
$
is in the Laguerre-P\'olya class and
\begin{align*}
  g_d^*\left( e^{-z^2};x\right) & = \sum_{j=0}^{\lfloor d/2 \rfloor} \binom{d}{2j} (-1)^j \frac{(2j)!}{j!} x^{d-j} \\
 & = \sum_{j=0}^{\lfloor d/2 \rfloor}  \frac{ d!(-1)^j }{(d-2j)! j!} x^{d-j} =H_d(x/2).
\end{align*}
By Lemma \ref{wim}, for any $\Phi(z)=\sum_{j=0}^\infty c_j z^j/j!$, we then find
\begin{equation*}
   g_d^*\left(\Phi \cdot e^{-z^2};x\right) = \sum_{j=0}^d \binom{d}{j} c_j \cdot H_{d-j}(x/2)
\end{equation*}
and in particular
\begin{equation} \label{tus}
  g_d^*\left(\Theta^{(n)} \cdot e^{-z^2};x\right) = \sum_{j=0}^d \binom{d}{j} \g(j+n) \cdot H_{d-j}(x/2) =  P^{d,n}(x/2)
\end{equation}
with our notation \e{jen-herm}.

\begin{proof}[Proof of Theorem \ref{hyp}]
As seen in the proof of Corollary \ref{cdo}, $\Theta^{(n)}(z)$ is  real and entire of genus $1^*$ for all $n\gqs 0$. Theorem \ref{prod} with $\Phi(z) = \Theta^{(n)}(z)$ and $\Omega(z)=e^{-z^2}$ implies that $\Theta^{(n)}(z)$ is hyperbolic if and only if $P^{d,n}(x)$ as in \e{tus} is hyperbolic for all $d\gqs 1$. As  in Corollary \ref{cdo2}, the Riemann hypothesis is equivalent to $\Theta^{(n)}(z)$ being hyperbolic for all $n\gqs 0$.
\end{proof}

We noted in the introduction that $J^{d,n}(x)$ being hyperbolic implies that  $P^{d,n}(x)$ is hyperbolic. This is shown next with the help of a general result of P\'olya from \cite[p. 242]{Po15}.

\begin{theorem} \label{prodj}
Let $\Phi(z)$  and $\Omega(z)$ be in the Laguerre-P\'olya class with $\Phi(z)=\sum_{j=0}^\infty c_j z^j/j!$. Suppose
\begin{equation} \label{ranxy}
   \sum_{j=0}^\infty \frac{c_j}{j!} \Omega^{(j)}(z)
\end{equation}
converges  when $|z|<\rho$ for some $\rho>0$. Then \e{ranxy} represents a function in the Laguerre-P\'olya class.
\end{theorem}

Following P\'olya,  the example of $\Omega(z) = e^{-z^2}$ may be used in Theorem \ref{prodj} along with the formula
\begin{equation*}
  \Omega^{(j)}(z) = (-1)^j e^{-z^2} H_j(z).
\end{equation*}
We obtain the following special case, which  was also stated in Lemma II of \cite{Tu59}.

\begin{cor} \label{poly-herm}
Let $\sum_{j=0}^m a_j z^j $ be a hyperbolic polynomial with real coefficients. Then the polynomial
$
   \sum_{j=0}^m a_j H_j(z)
$
is also hyperbolic.
\end{cor}

\begin{cor} \label{poly-herm2}
If $J^{d,n}(x)$ is hyperbolic  then $P^{d,n}(x)$ is hyperbolic.
\end{cor}
\begin{proof}
If $J^{d,n}(x)$ is hyperbolic then so is the reciprocal polynomial $x^d J^{d,n}(1/x)$. An application of Corollary \ref{poly-herm} to this reciprocal now shows that $ P^{d,n}(x)$
is hyperbolic.
\end{proof}

Note that taking $\Phi(z)= J^{d,n}(x)$  and $\Omega(z) = H_d(x/2)$ in Theorem \ref{prodj}  implies  the  similar polynomial
\begin{equation*}
  \sum_{j=0}^d \binom{d}{j}\g(n+j) \frac{H_{d-j}(x)}{(d-j)!}
\end{equation*}
is also hyperbolic if $J^{d,n}(x)$ is. This uses the equality $\frac{d^j}{dx^j} H_d(x/2) = \frac{d!}{(d-j)!} H_{d-j}(x/2)$.

\subsection{Further examples of Jensen polynomials}

Theorem \ref{prod}  with $\Phi(z)=\Theta^{(n)}(z)$ may be used to give different criteria for the Riemann hypothesis, introducing an interesting flexibility. For $\Omega(z)=1$ we obtain the polynomials $J^{d,n}(X)$ and for $\Omega(z)=e^{-z^2}$ we obtain  $P^{d,n}(X)$. We could use   $\cos(z)$ and  $\sin(z)$ for $\Omega(z)$ for example, but  we next focus on a family of Bessel functions.

For $\alpha \in \R$ the Bessel function of the first kind has the series representation
\begin{equation*}
  J_\alpha(z) = \sum_{m=0}^\infty \frac{(-1)^m}{m! \, \G(m+\alpha+1)} \left(\frac z2 \right)^{2m+\alpha}.
\end{equation*}
Then $z^{-\alpha} J_\alpha(2z)$ is entire and even and we may put
\begin{equation*}
  \mathcal J_\alpha(z) := z^{-\alpha/2} J_\alpha(2\sqrt{z}) = \sum_{m=0}^\infty \frac{(-1)^m z^m}{m! \, \G(m+\alpha+1)} .
\end{equation*}
A short calculation as in \cite[Sect. 3]{DY09} finds
\begin{equation*}
  g_d(\mathcal J_\alpha;x) = \frac{d!}{\G(d+\alpha+1)} L^{(\alpha)}_d(x)
\end{equation*}
with  the generalized Laguerre polynomials given by
\begin{equation}
  L^{(\alpha)}_d(x)  =  \frac{\G(d+\alpha+1)}{d!} \sum_{k=0}^d \binom{d}{k}\frac{(-1)^k x^k}{\G(k+\alpha+1)}
 =   \sum_{k=0}^d \binom{d+\alpha}{d-k}(-1)^k \frac{ x^k}{k!}. \label{lague}
\end{equation}
These  polynomials are orthogonal for $\alpha>-1$ and therefore have only real roots when $d\gqs 1$ for $\alpha$ in this range. In fact   the Laguerre polynomials are known to have real roots for $\alpha \gqs -2$. For $d\gqs 2$ and $\alpha<-2$, $L^{(\alpha)}_{d}(x)$ has non-real roots except possibly when $\alpha$ is an integer. See \cite[Eq. (5.2.1), Thm. 6.7.3]{Sz75} for these results. It follows from Theorem \ref{psthm} that $\mathcal J_\alpha$ is in the
Laguerre-P\'olya class for all $\alpha \gqs -2$. Applying Theorem \ref{prod} with $\Phi = \Theta^{(n)}$ and $\Omega = \mathcal J_\alpha$, and using the alternate identity \e{ave}, proves an extended criterion for the Riemann hypothesis involving  the Laguerre polynomials:

\begin{theorem} \label{lag}
The Riemann hypothesis is true if and only if the polynomials
\begin{equation*}
  Q^{d,n,\alpha}(x):=\sum_{j=0}^d \binom{d+\alpha}{j} \g(n+j) x^j   L^{(\alpha)}_{d-j}(x)
\end{equation*}
are  hyperbolic for all $d \gqs 1$, all $n \gqs 0$ and all real $\alpha \gqs -2$.
\end{theorem}

For example, taking $d=2$,  the discriminant of $Q^{2,n,\alpha}(x)$ is
\begin{equation*}
  (\alpha+2) \g(n)^2 +(\alpha+1)^2(\alpha+2)^2\Bigl[\g(n+1)^2 -\g(n)\g(n+2)  \Bigr],
\end{equation*}
and this is non-negative if $\alpha \gqs -2$ and the Tur\'an inequality \e{tu} holds.


We note that Farmer, in the  interesting preprint \cite{Far}, puts these types of Jensen polynomial approaches to the Riemann hypothesis into context in the recent literature and comments on their likelihood of success.

\section{The asymptotics of $I_{\alpha}(n)$} \label{lap}

Before treating \e{iafn} we first look at the simpler case
\begin{equation} \label{simpl}
  I_\alpha(n):=\int_1^\infty  (\log t)^n e^{-\alpha t}   \, dt
\end{equation}
which contains all the main ideas.
Recalling \e{covi} we define the rational functions
\begin{gather}
  a_r(v)  :=\sum_{j=0}^{2r} \frac{(2j+2r-1)!!}{j!} \left( \frac{v^2}{v+1}\right)^{j+r}
   \hat{B}_{2r,j}(\ell_3(v), \ell_4(v), \dots) \label{covi2}\\
\text{so that} \qquad  a_0(v) = 1,  \qquad a_1(v) =\frac{2 v^4+9 v^3+16 v^2+6 v+2}{24 (v+1)^3},  \qquad \text{etc}. \label{luex2}
\end{gather}

\begin{theorem} \label{ian}
Suppose $\alpha>0$ and set $u :=W(n/\alpha)$. Then as $n \to \infty$ we have
\begin{equation} \label{maini}
 I_{\alpha}(n) = \sqrt{2\pi} \frac{ u^{n+1}e^{u-n/u}}{\sqrt{(1+u)n}} \left( 1+  \sum_{r=1}^{R-1}\frac{a_r(u)}{n^r}+  O\left( \frac{u^R}{n^R}\right) \right)
\end{equation}
where the implied constant  depends only on $R \gqs 1$  and $\alpha$. Also $a_r(u) \ll u^r \ll \log^r(n)$.
\end{theorem}

\begin{proof}
We give the proof in the rest of this section, following Laplace's method as described in \cite[Sect. B6]{Fl09} for example. See also \cite[Thm. 9]{GORZ} and \cite[Sect. 2.4]{Rom} for similar arguments. Let $g(t)$ denote the integrand in \e{simpl}. Then
\begin{equation*}
  g'(t) = \left( \frac n{t\log t} -\alpha \right) g(t)
\end{equation*}
and $g'(t_0)=0$ for the unique $t_0>1$ satisfying $t_0 \log t_0 = n/\alpha$.
With \e{lamb} we have that
\begin{equation} \label{boo}
  W(n/\alpha)=W(t_0\log t_0) = \log t_0 \qquad \text{implies} \qquad t_0=e^{W(n/\alpha)}= \frac{n/\alpha}{W(n/\alpha)}.
\end{equation}
The largest contribution to the integral $I_\alpha(n)$ will be near the maximum of the integrand at $t=t_0 \approx n/(\alpha \log (n/\alpha))$. We develop the integrand about this point:
\begin{align}
  I_\alpha(n) & = g(t_0) \int_1^\infty  \exp\left(\log\left( \frac{g(t)}{g(t_0)}\right)\right)  \, dt \notag \\
   & = t_0 g(t_0) \int_{1/t_0-1}^\infty  \exp\left(\log\left( \frac{g((x+1)t_0)}{g(t_0)}\right)\right)   \, dx \notag \\
 & = t_0 g(t_0) \int_{1/t_0-1}^\infty  \exp\left(n\log\left( 1+\frac{\log(x+1)}{\log t_0}\right)-\alpha t_0 x\right) \, dx. \label{fay}
\end{align}
For $u=W(n/\alpha)$ the  integrand is
\begin{equation} \label{expp}
  \exp\left(n\left[\log\left( 1+\frac{\log(x+1)}{u}\right)-\frac x u\right]\right) =
\exp\left(n\left[ h_u(x)-\frac x u\right]\right)
\end{equation}
with $ h_u(x)$ from \e{hot} having a power series expansion
 for  $|x|\lqs 1/2$ and $u\gqs 2$, say, with  coefficients $\ell_i(u)$ given in \e{covi}.
After this point, the equality between $u$ and $W(n/\alpha)$ may be ignored and we  treat $u$ as a free parameter with $u \gqs 2$, or later $u \gqs 1$.

\begin{lemma} \label{jow}
Suppose  $|z|\lqs 1/2$ for  $z \in \C$ and assume $u\gqs 2$. Then for $k\gqs 0$  we have
\begin{gather}
  h_u(z)  = \sum_{i=1}^{k} \ell_i(u) z^i + R_k(u,z) \notag\\
\text{where} \qquad
 |\ell_i(u)| \lqs \frac{3}{u} \left( \frac 43 \right)^i, \qquad |R_k(u,z)| \lqs \frac{12}{u} \left( \frac 43 \right)^{k}|z|^{k+1}. \label{wher}
\end{gather}
\end{lemma}
\begin{proof}
From the simple inequality
$
  |\log(w+1)| \lqs 2|w|$ for $|w|\lqs 3/4$, $w\in \C$
we obtain
\begin{equation*}
  |h_u(w)| \lqs 4|w|/u  \quad \text{for} \quad  |w|\lqs 3/4, u \gqs 2.
\end{equation*}
We may now bound $h_u^{(i)}(0)$ and the Taylor remainder using Cauchy's estimates in the usual way, with the bound $|h_u(w)| \lqs 3/u$ for $|w|=3/4$.
\end{proof}

\begin{lemma}[Neglecting the tails] \label{oop}
Suppose $0<\delta\lqs 1/100$ and $u\gqs 2$. Then
\begin{equation} \label{aco}
  \int_{1/t_0-1}^\infty  \exp\left(n\left[h_u(x)-\frac x u\right]\right) \, dx =  \int_{-\delta}^\delta  \exp\left(n\left[h_u(x)-\frac x u\right]\right) \, dx + O\left(u \exp\left(-\frac{\delta^2 n}{4u}\right)\right)
\end{equation}
where the implied constant is absolute.
\end{lemma}
\begin{proof}
The elementary inequalities
\begin{equation*}
  \log(1+x)\lqs x \quad (x\gqs 0), \qquad \log(1+y)<y/2 \quad (y\gqs 3)
\end{equation*}
imply that
\begin{equation*}
  h_u(x) \lqs \log(1+x/u) <x/(2u)  \quad \text{for} \quad  x \gqs 3u.
\end{equation*}
Therefore
\begin{equation} \label{out}
  \int_{3u}^\infty  \exp\left(n\left[h_u(x)-\frac x u\right]\right) \, dx
 \lqs  \int_{3u}^\infty  \exp\left(-\frac {n x}{2u}\right) \, dx = \frac{2u}n e^{-3n/2}.
\end{equation}
By design $h_u(x)-x/u$ is increasing for $x<0$ and decreasing for $x>0$. Then for $x\gqs \delta$ we have
\begin{align*}
  h_u(x)-x/u & \lqs h_u(\delta)-\delta/u \\
  & \lqs -\frac{u+1}{2u^2} \delta^2+|R_2(u,\delta)|
 < -\frac{\delta^2}{2u}\left( 1- 50 \delta \right) \lqs -\frac{\delta^2}{4u}.
\end{align*}
We obtain the same bound for $x\lqs -\delta$ and so
\begin{equation*}
 \left( \int_{1/t_0-1}^{-\delta} + \int_{\delta}^{3u} \right) \exp\left(n\left[h_u(x)-\frac x u\right]\right) \, dx
 \lqs  (1+3u)\exp\left(-\frac{\delta^2 n}{4u}\right).
\end{equation*}
This is bigger than the bound in \e{out} and the lemma is proved.
\end{proof}

Put $C:= n (1+u)/u^2$ so that $n \cdot \ell_2(u) =-C/2$ by \e{luex}. It follows from Lemma \ref{jow} that
\begin{equation} \label{twc}
  \int_{-\delta}^\delta  \exp\left(n\left[h_u(x)-\frac x u\right]\right) \, dx =  \exp\left(O(n\delta^{k+1}) \right) \int_{-\delta}^\delta  e^{-C x^2/2}\exp\left( n \sum_{i=3}^{k} \ell_i(u) x^i \right) \, dx.
\end{equation}
Write the sum in \e{twc} as $y:=\sum_{i=3}^{k} \ell_i(u) x^i$. Then $n |y|\ll \delta^3 n/u$ by \e{wher}. We would like $n|y|$ to be small and so $\delta^3 n$ should tend to $0$ as $n \to \infty$. Also, with the error in \e{aco},   $\delta^2 n$ should  tend to $\infty$. This means choosing $\delta$ between $n^{-1/2}$ and $n^{-1/3}$. We now fix
\begin{equation*}
  \delta := n^{-2/5} \qquad \text{so that} \qquad \delta^2 n = n^{1/5}, \qquad  \delta^3 n = n^{-1/5}.
\end{equation*}
With this choice of $\delta$ (and assuming $k\gqs 2$) we have $\exp\left(O(n\delta^{k+1}) \right) = 1+O(n\delta^{k+1})$. Also, with this $\delta$, the integrand on the right of \e{twc}  is $\ll_k 1$ and therefore
\begin{equation} \label{twc2}
  \int_{-\delta}^\delta  \exp\left(n\left[h_u(x)-\frac x u\right]\right) \, dx =   \int_{-\delta}^\delta  e^{-C x^2/2}\exp\left( n \sum_{i=3}^{k} \ell_i(u) x^i \right) \, dx +O(n^{3/5-2k/5}).
\end{equation}

Without truncating the Taylor series \e{hot} of $h_u(x)$, we may write
\begin{align}
  \exp\left(\sum_{i=3}^\infty n \cdot \ell_i(u) x^i \right) & = \sum_{i=0}^\infty x^i \sum_{j=0}^i \hat{B}_{i,j}(\ell_3(u), \ell_4(u), \dots) \frac{n^j x^{2j}}{j!}\notag \\
& = \sum_{r=0}^\infty e_r(n,u) x^r \label{suraj}
\end{align}
for
\begin{equation} \label{esum}
  e_r(n,u) := \sum_{j=0}^{\lfloor r/3\rfloor} \hat{B}_{r-2j,j}(\ell_3(u), \ell_4(u), \dots) \frac{n^j}{j!}.
\end{equation}
It follows that $ e_r(n,u)$ is a polynomial in $n$ of degree at most $\lfloor r/3\rfloor$ with coefficients that are rational functions in $u$. Using \e{pobell3} and \e{wher} shows
\begin{equation}\label{and}
e_r(n,u) \ll_r n^{r/3}.
\end{equation}
The next lemma gives a truncated version of \e{suraj}.

\begin{lemma} \label{expx}
For all $x$ with $|x|\lqs \delta=n^{-2/5}$ and all $u\gqs 1$ we have
\begin{equation} \label{pat}
  \exp\left(\sum_{i=3}^k n \cdot \ell_i(u) x^i \right) = \sum_{r=0}^{k-1} e_r(n,u) x^r + O\left( \frac{1}{ n^{k/15}}\right)
\end{equation}
for an implied constant depending only on $k$.
\end{lemma}
\begin{proof}
For a number $c$ to be chosen later, the left side of \e{pat} is
\begin{align*}
  \exp(ny) & = \sum_{j=0}^c (ny)^j/j! + O\left((n|y|)^{c+1} \right) \\
& = \sum_{j=0}^c \frac{n^j x^{2j}}{j!} \left( \sum_{i=1}^{k-2}  \ell_{i+2}(u) x^i \right)^j + O\left((u \cdot n^{1/5})^{-c-1} \right) \\
& = \sum_{j=0}^c \frac{n^j x^{2j}}{j!}  \sum_{i=j}^{(k-2)j} \hat{B}_{i,j}(\ell_3(u), \ell_4(u), \dots, \ell_k(u),0,0, \dots) x^i + O\left(\frac{1}{ n^{(c+1)/5}} \right).
\end{align*}
If we define $e_r(n,u)_k$ as in \e{esum}, but with $\ell_j(u)$ inside the Bell polynomial replaced by $0$ when $j\gqs k+1$, then we obtain
\begin{equation*}
  \exp(ny) = \sum_{m=0}^{k-1} e_m(n,u) x^m + \sum_{m=k}^{k c} e_m(n,u)_k x^m + O\left(\frac{1}{ n^{(c+1)/5}} \right)
\end{equation*}
because $e_m(n,u)_k = e_m(n,u)$ for $m\lqs k-1$. As in \e{and} we have $e_m(n,u)_k \ll n^{m/3}$ and hence
\begin{equation*}
  \sum_{m=k}^{k c} e_m(n,u)_k x^m \ll \sum_{m=k}^{k c} n^{m/3} x^m
  \ll \sum_{m=k}^{k c} (\delta^3 n)^{m/3} \ll \sum_{m=k}^{k c}n^{-m/15}  \ll n^{-k/15}.
\end{equation*}
Choosing any $c\gqs k/3-1$ completes the proof.
\end{proof}


\begin{lemma}[Central approximation and completing the tails] We have
\begin{equation} \label{twc3}
  \int_{-\delta}^\delta  e^{-C x^2/2}\exp\left( n \sum_{i=3}^{k} \ell_i(u) x^i \right) \, dx
=\frac{\sqrt{2\pi} u}{\sqrt{(u+1)n}} \sum_{m=0}^{\lfloor(k-1)/2\rfloor} \frac{a^*_m}{n^m} + O\left( \frac{1}{ n^{k/15}}\right)
\end{equation}
when $\delta = n^{-2/5}$, $C= n (1+u)/u^2$, $u\gqs 1$ and
\begin{equation*}
  a^*_m=a^*_m(n,u):=  (2m-1)!!  \left( \frac{u^2}{u+1}\right)^m  \ e_{2m}(n,u).
\end{equation*}
\end{lemma}
\begin{proof}
Lemma \ref{expx} implies the left side of \e{twc3} equals
\begin{equation} \label{wworld}
 \sum_{r=0}^{k-1} e_r(n,u) \int_{-\delta}^\delta  e^{-C x^2/2} x^r \, dx +  O\left( \frac{1}{ n^{k/15}}\right).
\end{equation}
 If we extend the integral in \e{wworld} to all of $\R$ then the difference is twice
\begin{equation} \label{wworld2}
  \int_{\delta}^\infty  e^{-C x^2/2} x^r \, dx \lqs \int_{\delta}^\infty  e^{-C \delta x/2} x^r \, dx
.
\end{equation}
With \e{gmmb} and the inequality $C> n/u$, \e{wworld2} is bounded by a constant depending on $r$ times
\begin{equation*}
  \left(\frac {2 \delta^{r-1}}{C} + \left(\frac{2}{C\delta}\right)^{r+1} \right)  e^{-C \delta^2/2} \ll e^{-n^{1/5}/(2u)}.
\end{equation*}
For $m\in \Z_{\gqs 0}$ we use
\begin{equation*}
  \int_{-\infty}^\infty  e^{-C x^2/2} x^{2m} \, dx = \frac{\G(m+1/2)}{(C/2)^{m+1/2}} = \sqrt{\frac{2\pi}{C}}\frac{(2m-1)!!}{C^m}.
\end{equation*}
Therefore
\begin{multline*}
  \int_{-\delta}^\delta  e^{-C x^2/2}\exp\left( n \sum_{i=3}^{k} \ell_i(u) x^i \right) \, dx \\
  =  \sqrt{\frac{2\pi}{C}} \sum_{m=0}^{\lfloor(k-1)/2\rfloor} e_{2m}(n,u)\frac{(2m-1)!!}{C^m} + O\left( \frac{1}{ n^{k/15}} + n^{k/3} e^{-n^{1/5}/(2u)}\right)\\
=\frac{\sqrt{2\pi} u}{\sqrt{(u+1)n}} \sum_{m=0}^{\lfloor(k-1)/2\rfloor} \frac{a^*_m}{n^m} + O\left( \frac{1}{ n^{k/15}}\right)
\end{multline*}
as required.
\end{proof}

Each $a^*_m(n,u)$ is a polynomial in $n$ of degree at most $\lfloor 2m/3\rfloor$ with coefficients in $\Q(u)$.
Explicitly:
\begin{equation} \label{cumb}
  a^*_m =   \sum_{j=0}^{\lfloor 2m/3\rfloor} c_{m,j} n^j \quad \text{for} \quad c_{m,j}:= \frac{(2m-1)!!}{j!}  \left( \frac{u^2}{u+1}\right)^m \hat{B}_{2m-2j,j}(\ell_3(u), \ell_4(u), \dots).
\end{equation}
It is convenient to replace $k$ by $2Lk+1$ with an $L$ to be chosen later.
Simplifying $\sum_{m=0}^{Lk} a^*_m/n^m$ into a polynomial in $1/n$ we find
\begin{equation*}
  \sum_{m=0}^{Lk} \frac{a^*_m(n,u)}{n^m} =  \sum_{r=0}^{Lk} \frac{\hat{a}_r(u)}{n^r} \quad \text{for} \quad \hat{a}_r(u)
=\sum_{j=0}^{2r} c_{r+j,j}
\end{equation*}
provided  $c_{m,j}$ is set to zero for $m>Lk$. In other words
\begin{equation*}
  \hat{a}_r(u)=\sum_{j=0}^{\min\{2r,Lk-r\}} \frac{(2j+2r-1)!!}{j!} \left( \frac{u^2}{u+1}\right)^{j+r}
   \hat{B}_{2r,j}(\ell_3(u), \ell_4(u), \dots)
\end{equation*}
and equation \e{twc3} becomes
\begin{equation} \label{twc4}
  \int_{-\delta}^\delta  e^{-C x^2/2}\exp\left( n \sum_{i=3}^{2Lk+1} \ell_i(u) x^i \right) \, dx
=\frac{\sqrt{2\pi} u}{\sqrt{(u+1)n}} \sum_{r=0}^{Lk} \frac{\hat{a}_r(u)}{n^r} + O\left( \frac{1}{ n^{2Lk/15}}\right).
\end{equation}

Now  \e{twc2} and \e{twc4}  can be combined to produce
\begin{multline} \label{twc5}
  \int_{-\delta}^\delta  \exp\left(n\left[h_u(x)-\frac x u\right]\right) \, dx
=   \frac{\sqrt{2\pi} u}{\sqrt{(u+1)n}} \\
\times \left[\sum_{r=0}^{k-1} \frac{\hat{a}_r(u)}{n^r} + \sum_{r=k}^{Lk} \frac{\hat{a}_r(u)}{n^r}
+O\left( \frac{n^{1/2}}{ n^{(4Lk-1)/5}} +  \frac{n^{1/2}}{ n^{2Lk/15}}\right) \right] .
\end{multline}
Comparing $ \hat{a}_r(u)$ with $a_r(u)$ defined in \e{covi2}, we see they are equal when $r<k$ and $L\gqs 3$.
Also, note that $\ell_i(u) \ll 1/u$ by \e{wher} implies that $\hat{B}_{2r,j}(\ell_3(u), \ell_4(u), \dots) \ll 1/u^j$ by \e{pobell3}. Hence
\begin{equation} \label{puin}
  a_r(u), \ \hat{a}_r(u) \ll_r u^r.
\end{equation}
Choosing $L=15$, for example, in \e{twc5} shows
\begin{equation*}
   \int_{-\delta}^\delta  \exp\left(n\left[h_u(x)-\frac x u\right]\right) \, dx
=   \frac{\sqrt{2\pi} u}{\sqrt{(u+1)n}}  \left[\sum_{r=0}^{k-1} \frac{a_r(u)}{n^r}
+O\left( \frac{u^k}{ n^{k}}\right) \right] .
\end{equation*}
Inserting this into Lemma \ref{oop} and recalling \e{fay} gives
\begin{equation*}
  I_\alpha(n)  =  t_0 g(t_0)  \frac{\sqrt{2\pi} u}{\sqrt{(u+1)n}}  \left[\sum_{r=0}^{k-1} \frac{a_r(u)}{n^r}
+O\left( \frac{u^k}{ n^{k}}+ n^{1/2} u \exp\left(-\frac{ n^{1/5}}{4u}\right)\right) \right].
\end{equation*}
Finally $t_0=e^u$ and $g(t_0)=u^n e^{-n/u}$ by \e{boo} gives \e{maini} and completes the proof of Theorem \ref{ian}.
\end{proof}

\section{Generalizing Theorem \ref{ian}}

Define
\begin{equation*}
  I_\alpha(f;n):=\int_1^\infty  (\log t)^n e^{-\alpha t}  f(t) \, dt \qquad \qquad (n, \alpha >0).
\end{equation*}
As long as $f(t)$ is reasonably well-behaved and can be developed in a power series about $t=t_0$, with coefficients that are relatively small, then the proof of Theorem \ref{ian} should go through. The examples we have in mind for our applications are $f(t)=t^\beta$ and $f(t)=e^{-(\log(t))^2/16}$. For the latter,
\begin{equation} \label{hew}
  \frac{f(t(x+1))}{f(t)} = \exp\left(-\frac v8 \log(x+1) -\frac 1{16}\log^2(x+1) \right)
\end{equation}
with $t=e^v$. Treating $x$ as  complex, we see (as in the proof of Lemma \ref{jow}) that the right side of \e{hew} is bounded for $|x|\lqs 1/(2v)$ and $v \gqs 1$, say. Then  Taylor's theorem and the usual estimates show
\begin{equation*}
  \frac{f(t(x+1))}{f(t)} =  \sum_{m=0}^{k-1} f_m(v) x^m +O( v^k |x|^k) \qquad \text{for} \qquad |x| \lqs 1/(2v), \ k \in \Z_{\gqs 0}
\end{equation*}
and that $f_m(v) \ll v^{m}$. The  implied constants depend only on $k$ and $m$, respectively.

The coefficients $f_m(v)$ can be computed explicitly by combining the series for $\log(x+1)$ and $e^x$ with the Bell polynomials.
Similarly to \e{wash} we find
\begin{align*}
\exp\left(-\frac v8 \log(x+1) \right) & = \sum_{i=0}^\infty x^i  \sum_{j=0}^i \frac 1{j!} \left( -\frac v8\right)^j \hat{B}_{i,j}( 1,-{\textstyle \frac 12},{\textstyle \frac 13},\dots),\\
  \exp\left( -\frac 1{16}\log^2(x+1) \right)  & = \sum_{i=0}^\infty x^i  \sum_{j=0}^{\lfloor i/2 \rfloor} \frac 1{j!} \left( -\frac 1{16}\right)^j \hat{B}_{i,2j}( 1,-{\textstyle \frac 12},{\textstyle \frac 13},\dots).
\end{align*}
So with
\begin{equation}\label{pq}
  p_i(v):=\sum_{j=0}^i \frac 1{j!} \left( -\frac v8\right)^j \hat{B}_{i,j}( 1,-{\textstyle \frac 12},{\textstyle \frac 13},\dots),
\qquad
q_i := \sum_{j=0}^{\lfloor i/2 \rfloor} \frac 1{j!} \left( -\frac 1{16}\right)^j \hat{B}_{i,2j}( 1,-{\textstyle \frac 12},{\textstyle \frac 13},\dots),
\end{equation}
we find that $f_m(v) =\sum_{i=0}^m p_i(v) \cdot q_{m-i}$ is a polynomial in $v$ of degree $m$.

We may isolate the properties of  this example into a definition.

\begin{adef} \label{suit} A continuous function $f:[1,\infty) \to \R$ is {\em suitable} if  there exist   real constants $b, \lambda \gqs 0$ and  functions $f_m$ for $m \in \Z_{\gqs 0}$ so that the following conditions hold.
\begin{enumerate}
  \item For $t \gqs 1$ we have  $0< f(t) \ll t^b$.
  \item With $t=e^v$ and $v\gqs 1$ we have
\begin{equation} \label{ffd}
  \frac{f(t(x+1))}{f(t)} = \sum_{m=0}^{k-1} f_m(v) x^m +O( v^{\lambda k} |x|^k) \qquad (|x| \lqs 1/(2v^{\lambda}), \ k \in \Z_{\gqs 0}).
\end{equation}
  \item Lastly, $f_m(v) \ll v^{\lambda m}$ for all $m \in \Z_{\gqs 0}$.
\end{enumerate}
The implied constant  in (ii) depends only on $k$ and $f$; the one in (iii) depends only on $m$ and $f$.
\end{adef}


As we have seen, $f(t)=e^{-(\log(t))^2/16}$ is suitable with $b=0$ and $\lambda=1$.
The example $f(t)=t^\beta$, for any real $\beta$, is suitable with $b=\beta$, $f_m(v)=\binom{\beta}{m}$ and $\lambda =0$. It is also easy to show that products of suitable functions are  suitable. An example of a function that is not suitable is $f(t)=e^{-t}$ since $f_1(v)=-e^v$ is too large.

Define
\begin{equation} \label{same}
  a_r(f;v)  :=\sum_{j=0}^{2r} \frac{(2j+2r-1)!!}{j!} \left( \frac{v^2}{v+1}\right)^{j+r}
  \sum_{i=j}^{2r} \hat{B}_{i,j}(\ell_3(v), \ell_4(v), \dots) \cdot f_{2r-i}(v),
\end{equation}
with $\ell_m(v)$ given in \e{covi} as usual. Then the following theorem generalizes Theorem \ref{ian}, and reduces to it when $f(t)=1$.

\begin{theorem} \label{ian2}
Let $f$ be a suitable function associated with $b$, $\lambda$ and coefficients $f_m$.
Suppose $\alpha>0$ and set $u :=W(n/\alpha)$. Then as $n \to \infty$ we have
\begin{equation} \label{mainif}
 I_{\alpha}(f;n) = \sqrt{2\pi} \frac{ u^{n+1} f(e^u) e^{u-n/u}}{\sqrt{(1+u)n}} \left( 1+  \sum_{r=1}^{R-1}\frac{a_r(f;u)}{n^r}+  O\left( \frac{u^{R(1+2\lambda)}}{n^R}\right) \right)
\end{equation}
where  the implied constant  depends only on $R \gqs 1$, $\alpha$   and $b$. Also   $a_r(f;u)$ is defined  in \e{same} and satisfies $a_r(f;u) \ll u^{r(1+2\lambda)}$. 
\end{theorem}
\begin{proof}
We can reuse the proof of Theorem \ref{ian}. As in \e{fay},
\begin{equation*}
  I_{\alpha}(f;n) = t_0 g(t_0) \int_{1/t_0-1}^\infty \exp\left(n\left[h_u(x)-\frac x u\right]\right) f(t_0(x+1)) \, dx.
\end{equation*}
Using $f(t)\ll t^b$, $t_0=n/(\alpha u)$ and \e{gmmb} we see that
\begin{multline} \label{acox}
  \int_{1/t_0-1}^\infty  \exp\left(n\left[h_u(x)-\frac x u\right]\right) f(t_0(x+1))\, dx \\
=  \int_{-\delta}^\delta  \exp\left(n\left[h_u(x)-\frac x u\right]\right) f(t_0(x+1)) \, dx + O\left(u \cdot n^b \exp\left(-\frac{\delta^2 n}{4u}\right)\right)
\end{multline}
for $0<\delta\lqs 1/100$ and $u\gqs 2$, where the implied constant depends only on $\alpha$ and $b$. This is as in Lemma \ref{oop} with an extra factor $n^b$ in the error. For $e_r(n,u)$ defined in \e{esum}, set
\begin{equation} \label{var}
  \varepsilon_r(n,u):= \sum_{j=0}^r e_j(n,u) f_{r-j}(u).
\end{equation}

\begin{lemma} 
 For all $x$ with $|x|\lqs \delta=n^{-2/5}$  we have
\begin{equation*} 
  \exp\left(\sum_{i=3}^k n \cdot \ell_i(u) x^i \right) \frac{f(t_0(x+1))}{f(t_0)} = \sum_{r=0}^{k-1} \varepsilon_r(n,u) x^r + O\left( \frac{1}{ n^{k/15}}\right)
\end{equation*}
as $n\to \infty$ for an implied constant depending only on $k$ and $\lambda$.
\end{lemma}
\begin{proof}
Recall from \e{boo} that $t_0=e^u$ and $u=W(n/\alpha) = \log n - \log(\alpha u)$. So $u\gqs 1/\alpha$ implies $u \lqs \log n$. Hence, for $n$ large enough, $n^{-2/5} \lqs 1/(2 u^\lambda)$ and we may use \e{ffd} with $t=t_0$ and $v=u$.

Next note the bounds
\begin{equation}\label{pinf}
  \exp\left(\sum_{i=3}^k n \cdot \ell_i(u) x^i \right) = O(1), \qquad \frac{f(t_0(x+1))}{f(t_0)} = O(1)
\end{equation}
for $|x|\lqs \delta=n^{-2/5}$, where the left bound in \e{pinf} is shown after \e{twc} and the right bound follows from \e{ffd}. Therefore, using Lemma \ref{expx} and \e{ffd},
\begin{multline*}
  \exp\left(\sum_{i=3}^k n \cdot \ell_i(u) x^i \right) \frac{f(t_0(x+1))}{f(t_0)} = \left( \sum_{r=0}^{k-1} e_r(n,u) x^r\right)
\left(\sum_{m=0}^{k-1} f_m(u) x^m  \right)  + O\left( \frac{1}{ n^{k/15}}+ \frac{u^{\lambda k}}{ n^{2k/5}}\right)\\
= \sum_{r=0}^{k-1} \varepsilon_r(n,u) x^r
+ \sum_{r=k}^{2k-2} x^r \sum_{j=r-k+1}^{k-1} e_j(n,u) f_{r-j}(u)
+ O\left( \frac{1}{ n^{k/15}}\right).
\end{multline*}
Finally
\begin{multline*}
   \sum_{r=k}^{2k-2} x^r \sum_{j=r-k+1}^{k-1} e_j(n,u) f_{r-j}(u) \ll  \sum_{r=k}^{2k-2} \delta^r \sum_{j=r-k+1}^{k-1} n^{j/3} u^{\lambda(r-j)} \\
   \ll  \sum_{r=k}^{2k-2} \delta^r  n^{(k-1)/3} u^{\lambda(k-1)}
 \ll  \delta^k  n^{(k-1)/3} u^{\lambda(k-1)} =n^{-k/15-1/3}u^{\lambda(k-1)} \ll n^{-k/15},
\end{multline*}
where we used the bound \e{and} for $e_j(n,u)$. This completes the proof.
\end{proof}

It is clear from \e{and}, part (iii) of Definition \ref{suit} and \e{var} that
\begin{equation*}
  \varepsilon_r(n,u) \ll_r n^{r/3}.
\end{equation*}
We may now continue the proof of Theorem \ref{ian} with $e_r(n,u)$ replaced by $\varepsilon_r(n,u)$. In \e{cumb},
$a^*_m$ is replaced by $ a^*_m(f) =   \sum_{j=0}^{\lfloor 2m/3\rfloor} c_{m,j}(f) n^j$ for
\begin{equation*}
  c_{m,j}(f):= \frac{(2m-1)!!}{j!}  \left( \frac{u^2}{u+1}\right)^m \sum_{i=j}^{2m-2j}\hat{B}_{i,j}(\ell_3(u), \ell_4(u), \dots) f_{2m-2j-i}(u).
\end{equation*}
Then $a_r(u)$ becomes $a_r(f;u)=\sum_{j=0}^{2r} c_{r+j,j}(f)$ and so
\begin{equation} \label{puin2}
  a_r(f;u)  =\sum_{j=0}^{2r} \frac{(2j+2r-1)!!}{j!} \left( \frac{u^2}{u+1}\right)^{j+r}
  \sum_{i=j}^{2r} \hat{B}_{i,j}(\ell_3(u), \ell_4(u), \dots) f_{2r-i}(u)
\end{equation}
agreeing with \e{same}. The proof is finished as in Theorem \ref{ian} with the only difference  that the bound \e{puin} is replaced by   $a_r(f;u) \ll u^{r(1+2\lambda)}$ which follows from \e{puin2}.
\end{proof}

\section{Tur\'an's expansion of $\Xi(z)$} \label{tura}

For $t>0$ and $x\in \R$ set
\begin{equation} \label{mega}
  \theta(t) :=\sum_{k\in \Z} e^{-\pi k^2 t}, \qquad \omega(t):= \frac 12 \left(3t \theta'(t) + 2t^2 \theta''(t)  \right), \qquad \Phi(x):= 2e^{x/2}\omega(e^{2x}).
\end{equation}
Recall Tur\'an's coefficients $b_{2n}$ from \e{turanbn}.
As shown in \cite[Eq. (2.1), Thm. 2.1]{Rom}, they have the formulas
\begin{align*}
  b_{2n} & =\frac 1{2^{2n-1}(2n)!}\int_0^\infty x^{2n} e^{-x^2/4} \Phi(x)\, dx\\
& =\frac 1{2^{4n-1}(2n)!}\int_1^\infty (\log t)^{2n} \omega(t) t^{-3/4} e^{-(\log t)^2/16}\, dt.
\end{align*}

\begin{proof}[Proof of Theorem \ref{turanb}]
Write
\begin{equation*}
  \omega(t)= \sum_{m=1}^\infty \left(2\pi^2 m^4 t^2-3\pi m^2 t \right) e^{-\pi m^2 t} \qquad (t > 0).
\end{equation*}
For $t\gqs 1$, $\omega(t)$ is dominated by its first two terms and it is easy to show that
\begin{equation*}
  \omega(t) = \left(2\pi^2  t^2-3\pi  t \right) e^{-\pi  t} + O\left( t^2 e^{-4\pi t}\right) \qquad (t \gqs 1).
\end{equation*}
With
\begin{equation*}
  T_{\alpha,\beta}(n):= \int_1^\infty (\log t)^n e^{-\alpha t} t^{\beta} e^{-(\log t)^2/16}\, dt,
\end{equation*}
we obtain
\begin{equation} \label{v9}
  b_{2n}=\frac 1{2^{4n-1}(2n)!} \left(2\pi^2  T_{\pi,5/4}(2n) -3\pi T_{\pi,1/4}(2n)  +O\left( T_{4\pi,5/4}(2n)\right)\right).
\end{equation}
 Theorem \ref{ian2}  may be used to find the asymptotics of $ T_{\alpha,\beta}(n)$ since, by the discussion around Definition \ref{suit}, $f(t) = t^{\beta} e^{-(\log t)^2/16}$ is suitable with $b=\beta$, $\lambda = 1$ and
\begin{equation} \label{pq2}
  f_m(v) = \sum_{j_1+j_2+j_3 = m} \binom{\beta}{j_1} p_{j_2}(v) q_{j_3},
\end{equation}
using the formulas \e{pq}. We write $a_r(f;v)_\beta$ for the corresponding coefficients \e{same}, adding a subscript to keep track of the parameter $\beta$. Then
for $u :=W(n/\alpha)$ we have
\begin{equation} \label{bor}
 T_{\alpha,\beta}(n) = \sqrt{2\pi} \frac{ u^{n+1} e^{u(\beta+1)-n/u-u^2/16}}{\sqrt{(1+u)n}} \left( 1+  \sum_{r=1}^{R-1}\frac{a_r(f;u)_\beta}{n^r}+  O\left( \frac{u^{3R}}{n^R}\right) \right)
\end{equation}
as $n \to \infty$. For  $2\pi^2  T_{\pi,5/4}(2n) -3\pi T_{\pi,1/4}(2n)$, $u$ becomes $w:=W(2n/\pi)$ and in the main term of \e{bor} we have $e^{w(\beta+1)}$ for $\beta =5/4,1/4$. Using the identity $e^w=2n/(\pi w)$ gives
\begin{equation} \label{dent}
  e^{w(1/4+1)} = e^{w(5/4+1)}\frac{\pi w}{2n}
\end{equation}
which lets us add the expressions. The result is
\begin{multline} \label{bor2}
  2\pi^2  T_{\pi,5/4}(2n) -3\pi T_{\pi,1/4}(2n) \\
  =2 \pi^{5/2}\frac{ w^{2n+1}}{\sqrt{(w+1) n}} e^{9w/4-2n/w-w^2/16}
 \left( 1+  \sum_{r=1}^{R-1}\frac{\tau_r(w)}{n^r}+  O\left( \frac{w^{3R}}{n^R}\right) \right)
\end{multline}
where
\begin{equation} \label{pq3}
  \tau_r(w) := 2^{-r-1}\left(2 a_r(f;w)_{5/4}-3w \cdot a_{r-1}(f,w)_{1/4}\right) \ll w^{3r}.
\end{equation}

Comparing \e{bor2} with the error
\begin{equation*}
  T_{4\pi,5/4}(2n) \ll \frac{ \tilde{w}^{2n+1}}{\sqrt{(\tilde{w}+1) n}} e^{9\tilde{w}/4-2n/\tilde{w}-\tilde{w}^2/16} \qquad  \text{for}
\qquad \tilde{w}:=W(2n/(4\pi)),
\end{equation*}
in \e{v9}, we see the ratio is bounded by
\begin{equation} \label{stot}
  \frac{\tilde{w}}{w} \sqrt{\frac{w+1}{\tilde{w}+1}} \exp\left(\frac{9(\tilde{w}-w)}{4}-\frac{\tilde{w}^2-w^2}{16}\right)
\left( \frac{\tilde{w}}{w} e^{1/w-1/\tilde{w}}\right)^{2n}.
\end{equation}
The inequalities  needed to estimate \e{stot} are developed next; they will also be used in Sect. \ref{hyp-jen-poly}.

Beginning with $1+x \lqs e^x$ for $x\gqs 0$, we obtain
\begin{equation}\label{ineqx}
  \log(1+x) \lqs x \qquad (x\gqs 0)
\end{equation}
and also $(1+x)^{1/x}\lqs e$ for $x>0$. Hence
\begin{equation}\label{ineqy}
  \left( 1+\frac ab \right)^b \lqs e^a \qquad (a\gqs 0, \ b>0).
\end{equation}

\begin{lemma} \label{woo}
For $x> 0$ and $c\gqs 1$ we have
\begin{equation} \label{oreg}
 \left(1-\frac 1{W(x)}\right)\log c \lqs W(cx)-W(x)\lqs \log c.
\end{equation}
\end{lemma}
\begin{proof}
By taking logs, the identities $W(x) e^{W(x)}=x$ and   $W(cx) e^{W(cx)}=c x$ imply
\begin{align}
  \log W(x)+ W(x) & = \log x, \label{wbnd}\\
  \log W(cx)+ W(cx) & = \log x + \log c. \notag
\end{align}
Subtracting the first equation from the second gives
\begin{equation*}
  W(cx)-W(x) - \log c = \log \frac{W(x)}{W(c x)}
\end{equation*}
and the right side is $\lqs 0$ since the Lambert function is increasing. This proves the right inequality in \e{oreg}. Then
\begin{equation*}
  \log  \frac{W(c x)}{W(x)} \lqs \log  \frac{W( x)+\log c}{W(x)} = \log\left( 1+ \frac{\log c}{W(x)}\right) \lqs \frac{\log c}{W(x)}
\end{equation*}
where we used \e{ineqx}. The left inequality in \e{oreg} is then a result of:
\begin{equation*}
  - \frac{\log c}{W(x)} \lqs -\log  \frac{W(c x)}{W(x)} = \log  \frac{W(x)}{W(c x)} =W(cx)-W(x) - \log c. \qedhere
\end{equation*}
\end{proof}

Further inequalities for the Lambert function are contained in  \cite{Hoor}, for example.
Also note that \e{wbnd} and $W(e)=1$ imply that  $W(x)\lqs \log x$ for all $x\gqs e$.

\begin{lemma} \label{www}
For $w=W(2n/\pi)$, $\tilde{w}=W(2n/(4\pi))$ the expression \e{stot} is   $ \ll e^{-2n/\log n}$ for large $n$.
\end{lemma}
\begin{proof}
Lemma \ref{woo} implies that $(1-1/\tilde{w})\log 4 \lqs w-\tilde{w}\lqs \log 4$. If we write $\tilde{w} = w-\delta$ then,  by choosing $n$ large enough, we can certainly ensure $1< \delta <2$.
We have
\begin{equation*}
   \frac{\tilde{w}}{w} \sqrt{\frac{w+1}{\tilde{w}+1}} \exp\left(\frac{9(\tilde{w}-w)}{4}-\frac{\tilde{w}^2-w^2}{16}\right)
\ll \exp\left(-\frac{9\delta}{4}+\frac{\delta(2w-\delta)}{16}\right) \ll e^{\delta w/8}< e^{w/4}.
\end{equation*}
Also
\begin{equation*}
  w=W(2n/\pi) \lqs \log(2n/\pi)<\log n,
\end{equation*}
so that $ e^{w/4}< n^{1/4}$. For the remaining factor in \e{stot}
\begin{equation} \label{h7}
  \left( \frac{\tilde{w}}{w} e^{1/w-1/\tilde{w}}\right)^{2n} = \left( 1-\frac{\delta}{w}\right)^{w \cdot 2n/w} \exp\left( -\frac{2\delta n}{w \tilde{w}} \right).
\end{equation}
The bound
\begin{equation*}
  \left( 1-\frac{\delta}{w}\right)^{w} \lqs e^{-\delta} \qquad (w>0, 0\lqs \delta/w <1)
\end{equation*}
follows by taking logs and using the inequality $\log(1-x)\lqs -x$ for $0\lqs x < 1$. Therefore \e{h7} is at most
\begin{equation*}
  \exp\left(-\frac{2\delta n}{w} -\frac{2\delta n}{w^2}  \right) <  \exp\left(-\frac{2 n}{\log n} -\frac{2 n}{\log^2 n}  \right).
\end{equation*}
Simplifying the final bound with $n^{1/4} e^{-2n/\log^2 n} \ll 1$ completes the proof of the lemma.
\end{proof}

We have shown that the error term in \e{v9} is exponentially small compared to the main term, and so it may be included inside the error in \e{bor2}. Simplifying with $e^{w/2} = \sqrt{2n/(\pi w)}$ gives the final result
\begin{equation} \label{upo}
  b_{2n}=4\pi^{2}\frac{ e^{7w/4-w^2/16}}{(2n)!}\sqrt{\frac{ 2w}{w+1}} \left( \frac{w}{4e^{1/w}}\right)^{2n}
 \left( 1+  \sum_{k=1}^{K-1}\frac{\tau_k(w)}{n^k}+  O\left( \frac{w^{3K}}{n^K}\right) \right)
\end{equation}
and this completes the proof of Theorem \ref{turanb}.
\end{proof}

The main term of \e{upo} agrees with \cite[Thm. 2.7]{Rom}. We could also replace $1/(2n)!$ in \e{upo} with its asymptotic expansion  \e{gmx2}. The numbers $\tau_r(w)$ are given explicitly by \e{pq3}, using \e{pq}, \e{pq2} to define its components $f_m(v)$ and then $a_r(f;v)_\beta$ with \e{same}. For example
\begin{equation*}
  \tau_1(w) = -\frac{-3 w^6+78 w^5+217 w^4+468 w^3+284 w^2-32}{768 (w+1)^3}.
\end{equation*}
An example of the accuracy of Theorem \ref{turanb} for $2n=2000$ and different values of $K$ is displayed in Table \ref{jeb}.

\begin{table}[ht]
\centering
\begin{tabular}{ccc}
\hline
 $K$   & Theorem \ref{turanb} & \\
\hline
 $1$    & $2.37\textcolor{gray}{86738117568138992} \times 10^{-5738}$ & \\
 $3$    & $2.373211179\textcolor{gray}{9604212549} \times 10^{-5738}$ & \\
 $5$    & $2.373211179182932\textcolor{gray}{4664} \times 10^{-5738}$ & \\
 $7$    & $2.3732111791829329059\textcolor{gray}{} \times 10^{-5738}$  &  \\
\hline
\phantom{$b_{2000}$} & $ 2.3732111791829329059 \times 10^{-5738}$ & $b_{2000}$\\
\hline
\end{tabular}
\caption{The approximations of Theorem \ref{turanb} to $b_{2000}$.} \label{jeb}
\end{table}

\section{The Riemann $\xi$ function}

Recall the definitions of $\omega$ and $\Phi$ in \e{mega}.
The identity
\begin{equation} \label{mom2}
  \xi^{(2n)}(1/2) = 2^{1-2n} \int_1^\infty  (\log t)^{2n} \omega(t) t^{-3/4}  \, dt
\end{equation}
follows easily from Riemann's formulas of 1859. It also follows from the elegant expression
\begin{equation} \label{mom}
   \xi^{(n)}(1/2) = \int_{-\infty}^\infty  \Phi(y) \cdot  y^{n} \, dy \qquad \qquad (n \in \Z_{\gqs 0})
\end{equation}
by changing variables and recalling that $\Phi(-y) = \Phi(y)$; see \cite[Eqs. (1.9), (6.1)]{Rom}. The identity \e{mom} is implicit in Jensen's formula from \cite[p. 189]{Je13}:
\begin{equation*}
  g_d^*(\Xi;x)= \sum_{j=0}^d \binom{d}{j} i^j \xi^{(j)}(1/2) \cdot x^{d-j} = \int_{-\infty}^\infty   \Phi(y) \cdot (x+iy)^{d} \, dy.
\end{equation*}

\begin{proof}[Proof of Theorem \ref{mainthm2}]
We may follow the proof of Theorem \ref{turanb} in Sect. \ref{tura}. Setting
\begin{equation*}
  I_{\alpha,\beta}(n):= \int_1^\infty (\log t)^n e^{-\alpha t} t^{\beta} \, dt,
\end{equation*}
we obtain from \e{mom2}, as in \e{v9},
\begin{equation} \label{v9x}
  \xi^{(2n)}(1/2)= 2^{1-2n} \left(2\pi^2  I_{\pi,5/4}(2n) -3\pi I_{\pi,1/4}(2n)  +O\left( I_{4\pi,5/4}(2n)\right)\right).
\end{equation}
We noted after Definition \ref{suit} that $f(t)=t^\beta$  is suitable with $b=\beta$, $f_m(v)=\binom{\beta}{m}$ and $\lambda =0$. From \e{same} we find
\begin{equation} \label{same2}
  a_r(f;v)_\beta  =\sum_{j=0}^{2r} \frac{(2j+2r-1)!!}{j!} \left( \frac{v^2}{v+1}\right)^{j+r}
  \sum_{i=j}^{2r} \hat{B}_{i,j}(\ell_3(v), \ell_4(v), \dots)  \binom{\beta}{2r-i}
\end{equation}
and by Theorem \ref{ian2}
\begin{equation} \label{borx}
 I_{\alpha,\beta}(n) = \sqrt{2\pi} \frac{ u^{n+1} e^{u(\beta+1)-n/u}}{\sqrt{(1+u)n}} \left( 1+  \sum_{r=1}^{R-1}\frac{a_r(f;u)_\beta}{n^r}+  O\left( \frac{u^{R}}{n^R}\right) \right)
\end{equation}
as $n \to \infty$ for $u :=W(n/\alpha)$.
The asymptotics for the terms
$ I_{\pi,5/4}(2n)$ and $I_{\pi,1/4}(2n)$ may be combined, as in \e{bor2} using the identity \e{dent}, to produce
\begin{equation} \label{bor2x}
  2\pi^2  I_{\pi,5/4}(2n) -3\pi I_{\pi,1/4}(2n)
  =2 \pi^{5/2}\frac{ w^{2n+1} e^{9w/4-2n/w}}{\sqrt{(w+1) n}}
 \left( 1+  \sum_{r=1}^{R-1}\frac{\mu_r(w)}{n^r}+  O\left( \frac{w^{R}}{n^R}\right) \right)
\end{equation}
where $w:=W(2n/\pi)$ and
\begin{equation} \label{pq3x}
  \mu_r(w) := 2^{-r-1}\left(2 a_r(f;w)_{5/4}-3w \cdot a_{r-1}(f,w)_{1/4}\right) \ll w^{r}.
\end{equation}

Comparing \e{bor2x} with the error
\begin{equation*}
  I_{4\pi,5/4}(2n) \ll \frac{ \tilde{w}^{2n+1}}{\sqrt{(\tilde{w}+1) n}} e^{9\tilde{w}/4-2n/\tilde{w}} \qquad  \text{for}
\qquad \tilde{w}:=W(2n/(4\pi)),
\end{equation*}
in \e{v9x}, we see the ratio is bounded by
\begin{equation} \label{stotx}
  \frac{\tilde{w}}{w} \sqrt{\frac{w+1}{\tilde{w}+1}} \exp\left(\frac{9(\tilde{w}-w)}{4}\right)
\left( \frac{\tilde{w}}{w} e^{1/w-1/\tilde{w}}\right)^{2n}.
\end{equation}
Then \e{stotx} is smaller than \e{stot} and so $ \ll e^{-2n/\log n}$ for large $n$ by Lemma \ref{www}.
This exponentially small error may therefore be included in the error in \e{bor2x}.
Rearranging slightly with $e^{w/2} = \sqrt{2n/(\pi w)}$ completes the proof of Theorem \ref{mainthm2}.
\end{proof}

The coefficient $ \mu_1(w)$ may be computed as
\begin{equation*}
  \mu_1(w) = -\frac{w^4+66 w^3+53 w^2-8}{192 (w+1)^3}.
\end{equation*}

\begin{proof}[Proof of Theorem \ref{mainthm3}]
Recall that
\begin{equation} \label{gamma}
  \g(n):=  \frac{n!}{(2n)!}\xi^{(2n)}(1/2).
\end{equation}
The asymptotic expansion of the gamma function goes back to Laplace, with
\begin{equation} \label{gmx}
  \G(n+1)= \sqrt{2\pi n}\left(\frac{n}{e}\right)^n \left(1+\frac{\kappa_1}{n}+\frac{\kappa_2}{n^2}+ \cdots + \frac{\kappa_{k}}{n^{k}} +O\left(\frac{1}{n^{k+1}}\right)\right)
\end{equation}
often called Stirling's formula.
There are many ways to express the coefficients $\kappa_m$ and a simple example is
\begin{equation} \label{stx}
  \kappa_m = (2m-1)!! \sum_{j=0}^{2m}  \binom{-m-1/2}{j} \hat{B}_{2m,j}\left(-\frac{2}{3},\frac{2}{4},-\frac{2}{5},\frac{2}{6}, \cdots \right)
\end{equation}
from \cite[Eq. (8.1)]{OSper}. The same coefficients appear in the expansion of the reciprocal of gamma but with alternating signs. One way to see this is to note that the well-known asymptotic series for $\log \G(n+1)$ involves an odd function; see also \cite[Sect. VIII.3]{Fl09}.
Then we have
\begin{equation} \label{gmx2}
  \frac 1{\G(n+1)}= \frac 1{\sqrt{2\pi n}}\left(\frac{e}{n}\right)^n \left(1-\frac{\kappa_1}{n}+\frac{\kappa_2}{n^2}- \cdots + (-1)^{k}\frac{\kappa_{k}}{n^{k}} +O\left(\frac{1}{n^{k+1}}\right)\right).
\end{equation}

Let
\begin{equation} \label{stx2}
  \kappa^*_k := \sum_{j=0}^k (-2)^{-j} \kappa_j \cdot \kappa_{k-j}, \qquad
c_k(w):= \sum_{j=0}^k  \mu_j(w) \cdot \kappa^*_{k-j}.
\end{equation}
Formulas \e{gmx}, \e{gmx2} imply
\begin{equation} \label{n2}
  \frac{n!}{(2n)!} =\frac 1{\sqrt{2}}\left(\frac{e}{4n}\right)^n \left(1+\frac{\kappa^*_1}{n}+\frac{\kappa^*_2}{n^2}+ \cdots + \frac{\kappa^*_{k}}{n^{k}} +O\left(\frac{1}{n^{k+1}}\right)\right).
\end{equation}
Using \e{n2} and Theorem \ref{mainthm2} in \e{gamma} then proves Theorem \ref{mainthm3}.
\end{proof}

The coefficients $c_k(w)$ are completely described by the formulas \e{covi},  \e{same2},  \e{pq3x}, \e{stx} and \e{stx2}.
The first was given in \e{c1w} and the next one is
\begin{equation*}
  \textstyle c_2(w) = -\frac{1295 w^8+7804 w^7+21682 w^6+40124 w^5+29911 w^4+13712 w^3+2080
   w^2-768 w-256}{73728 (w+1)^6}.
\end{equation*}

\section{Zeros of $P^{d,n}(X)$} \label{hyp-jen-poly}

In this section we develop a condition for $P^{d,n}(X)$ to have only real roots.
It is based on the next result, due to Tur\'an from \cite[Thm. III]{Tu59}.

\begin{theorem} Let $G(z)=\sum_{j=0}^d c_j H_j(z)$ for real numbers $c_j$ and $H_j(z)$ the $j$th Hermite polynomial. Then the zeros of $G(z)$ are real and simple when
\begin{equation} \label{toy}
  \sum_{j=0}^{d-2} 2^j j! c_j^2 < 2^d(d-1)! c_d^2.
\end{equation}
\end{theorem}

\begin{proof}[Proof of Theorem \ref{chem}]
The condition \e{toy} for $P^{d,n}(X)$ to be hyperbolic, with $c_j=\binom{d}{j} \g(n+d-j)$, is
\begin{equation} \label{bolic}
  \sum_{j=2}^d \frac{d}{2^j j!} \binom{d}{j} \frac{\g(n+j)^2}{\g(n)^2} < 1.
\end{equation}
It follows from Theorem \ref{mainthm3} with $K=1$ and $w :=W(2n/\pi)$ that
\begin{equation*}
 \g(n) = \Psi(n) \left( 1+  O\left( \frac{\log(n)}{n}\right) \right) \qquad \text{for} \qquad  \Psi(n):= 4\pi^{2}e^{7w/4}  \sqrt{\frac {w}{w+1}} \left(\frac{e w^2}{16 n e^{2/w}} \right)^n.
\end{equation*}
 As a consequence, for any $\epsilon>0$ there exists an $N$ so that
\begin{multline} \label{fut}
   \frac{\g(n+j)}{\g(n)} \lqs
\frac{\Psi(n+j)}{\Psi(n)} (1+\epsilon)
=
\sqrt{ \frac{1+1/w}{1+1/w_j}}
 \left(\frac{ w_j^2}{w^2} \frac{ n}{n+j}  \right)^n \\
\times \exp\left(\frac{7(w_j-w)}{4}+ \frac{2n(w_j-w)}{w w_j}\right) \left(\frac{e w_j^2}{16 (n+j) e^{2/w_j}} \right)^j (1+\epsilon)
\end{multline}
for all $n\gqs N$ and all $j\gqs 0$, where we set $w_j :=W(2(n+j)/\pi)$.

\begin{lemma} \label{nba}
Suppose $c\gqs 1$. For $n$ large enough and $0\lqs j \lqs n^c-n$, we have
\begin{equation*}
   \frac{\g(n+j)}{\g(n)} <  1.01 \left(\frac{3 c^2\log^2(n)}{ 16 n} \right)^j.
\end{equation*}
\end{lemma}
\begin{proof}
By  Lemma \ref{woo} we have $w_j-w\lqs \log(1+\lambda)$ for $\lambda:=j/n$. Examining the components of \e{fut} we first find, using \e{ineqx}, \e{ineqy},
\begin{equation*}
   \left(\frac{ w_j}{w}  \right)^{2n} \lqs  \left(\frac{ w +\log(1+\lambda)}{w}  \right)^{2n}
= \left(1+\frac{ \log(1+\lambda)}{w}  \right)^{w \cdot 2n/w} \lqs e^{\log(1+\lambda) \cdot 2n/w} \lqs e^{2j/w}.
\end{equation*}
Next,
\begin{align*}
   \exp\left(\frac{7(w_j-w)}{4}+ \frac{2n(w_j-w)}{w w_j}\right) & \lqs  \exp\left(\frac{7\log(1+\lambda)}{4}+ \frac{2n \log(1+\lambda)}{w^2}\right) \\
& \lqs  \exp\left(\frac{7j}{4n}+ \frac{2j}{w^2}\right).
\end{align*}
Bounding trivially gives
\begin{equation*}
  \sqrt{ \frac{1+1/w}{1+1/w_j}}  \left( \frac{ n}{n+j}  \right)^n \lqs \sqrt{ 1+1/w}.
\end{equation*}
Assembling our bounds, and with $n$ large enough that $ (1+\epsilon) \sqrt{ 1+1/w} \lqs 1.01$, we have shown
\begin{equation} \label{yoho}
   \frac{\g(n+j)}{\g(n)}  \lqs 1.01 \left(\frac{ w_j^2}{16 (n+j)} \exp\left[\frac{7}{4n}+\frac 2w+  \frac{2}{w^2} - \frac{2}{w_j} +1\right] \right)^j.
\end{equation}
For $n$ sufficiently large the exponential term  can be made close to $e<3$ and  \e{yoho} is at most
\begin{equation*}
   1.01 \left(\frac{ 3 w_j^2}{16 (n+j)} \right)^j < 1.01 \left(\frac{ 3\log^2(2n^c/\pi)}{16 n} \right)^j. \qedhere
\end{equation*}
\end{proof}

Applying Lemma \ref{nba} with $c=2$, (it will become apparent that any $c\gqs 4/3$ will do), the left side of \e{bolic} is
\begin{align}
  \sum_{j=2}^d \frac{d}{2^j j!} \binom{d}{j} \frac{\g(n+j)^2}{\g(n)^2} & <  \sum_{j=2}^d \frac{d}{2^j j!} \frac{d^j}{j!}
\cdot  1.03 \left(\frac{ 9\log^4(n)}{ 16 n^2} \right)^j \notag\\
   & =  1.03 \sum_{j=2}^d \frac{d}{ (j!)^2}
  \left(\frac{9 d\log^4(n)}{ 32 n^2} \right)^j.\label{bla}
\end{align}
If we choose $n$ large enough so that
\begin{equation*}
  \left(\frac{9 d\log^4(n)}{ 32 n^2} \right)^2 \lqs \frac 3d
\end{equation*}
then \e{bla} is at most
\begin{equation*}
1.03 \sum_{j=2}^d \frac{d}{ (j!)^2}
  \left(\frac{\sqrt{3}}{ \sqrt{d}} \right)^j < 1.03 \sum_{j=2}^\infty \frac{3^{j/2}}{ (j!)^2} <0.94<1.
\end{equation*}
This verifies condition \e{bolic} and  Theorem \ref{chem} follows.
\end{proof}

{\small
\bibliography{xi-bib}
}

{\small 
\vskip 5mm
\noindent
\textsc{Dept. of Math, The CUNY Graduate Center, 365 Fifth Avenue, New York, NY 10016-4309, U.S.A.}

\noindent
{\em E-mail address:} \texttt{cosullivan@gc.cuny.edu}
}

\end{document}